\documentclass[12pt,reqno]{amsart}
\usepackage{amsmath}
\usepackage{amssymb}
\usepackage{graphicx}
\usepackage{tikz-cd}
\usepackage[T1]{fontenc}
\usepackage[utf8]{inputenc}

\usepackage[colorlinks=true,
            citecolor={teal},
            linkcolor=blue,
            urlcolor=blue]{hyperref}

\usepackage[all]{xy}
\usepackage{mathtools}

\usepackage{xspace}
\usepackage{bm}
\usepackage{amsmath}
\usepackage{amstext}
\usepackage{amsfonts}
\usepackage[mathscr]{euscript}
\usepackage{amscd}
\usepackage{latexsym}
\usepackage{amssymb}
\usepackage{enumerate}
\usepackage{xcolor}

\usepackage{anysize}
\marginsize{2.5cm}{2.5cm}{2.5cm}{2.5cm}
\usepackage{setspace}

\theoremstyle{plain}
    \newtheorem{theorem}{Theorem}[section]
    \newtheorem{proposition}[theorem]{Proposition}
    \newtheorem{lemma}[theorem]{Lemma}

    \newtheorem{corollary}[theorem]{Corollary}
    
    \newtheorem{subsec}[theorem]{}
    
    \newtheorem*{thma}{Theorem A}
    \newtheorem*{thmb}{Theorem B}

\theoremstyle{definition}
    
    \newtheorem{definition}[theorem]{Definition}

    \newtheorem{notation}[theorem]{Notation}

\theoremstyle{remark}
        \newtheorem{remark}[theorem]{Remark}

\renewcommand{\thefigure}{\arabic{section}.\arabic{figure}}

\newenvironment{myeq}[1][]
{\stepcounter{figure}\begin{equation}\tag{\thefigure}{#1}}
{\end{equation}}

\newcommand{\eat}[1]{}

\allowdisplaybreaks

\title[Degrees of Maps between Generalized Dold Manifolds]{Degrees of Maps between Generalized Dold Manifolds}
\author[Manas Mandal]{Manas Mandal}
\address{Department of Mathematics, IIT Kanpur, Kanpur 208016, India}
\email{manas.imsc@gmail.com}

\thanks{ }

\begin{document}

\begin{abstract}
We study the existence of maps of nonzero Brouwer degree between two orientable, equal-dimensional generalized Dold manifolds (GDMs) fibered by complex partial flag manifolds over real projective spaces. The GDMs considered here are obtained as orbit spaces of the diagonal involution on the product of a sphere and a complex partial flag manifold, acting antipodally on the sphere and by complex conjugation on the flag manifold. We show that no map of nonzero degree exists between two distinct GDMs when their base real projective spaces are different, except in one exceptional case. Furthermore, when the base real projective spaces coincide but fibers are distinct, we prove the nonexistence of maps of nonzero degree in most cases, including those in which at least one of the fibers is not a Grassmannian. As an application of our study, we establish cohomological rigidity for the class of products of a sphere with a complex or quaternionic partial flag manifold.
\end{abstract}

\maketitle

\tableofcontents

\section{Introduction}

The degree of a continuous map between closed, connected, oriented manifolds of the same dimension is an integer that measures how many times the domain wraps around the codomain. The notion is due to Brouwer \cite{Bro11}.  This is a vast generalization of the winding number of a continuous map between circles. A central problem in algebraic topology is determining the dominance order on manifolds: given two manifolds $M$ and $N$, one says that $M$ dominates $N$ if there exists a continuous map $f: M \to N$ of nonzero degree \cite{Gro82, Gro99}. The existence of such a map establishes a topological hierarchy, forcing the rational cohomology ring of the codomain to embed as a subring into that of the domain via the induced monomorphism $f^*: H^*(N; \mathbb{Q}) \to H^*(M; \mathbb{Q})$.

A well-known phenomenon in geometry and topology is that homogenous spaces, such as Grassmannians and flag manifolds, exhibit strong rigidity to admit nonzero degree maps.
Classic results by Paranjape-Srinivas, as well as Ramani-Sankaran and Sankaran-Sarkar  established that nonzero degree self-maps or maps between same-dimensional Grassmannians are severely constrained \cite{PS89, RS97, SS09}. In a recent work, similar constraints are observed for complex or quaternionic partial flag manifolds, demonstrating that continuous maps between distinct complex and quaternionic partial flag manifolds (one of which is not Grassmannian) must have degree zero \cite{Man26}.

The purpose of this paper is to extend this line of investigation to the setting of generalized Dold manifolds. Classical Dold manifolds were introduced by Dold in \cite{Do56} in order to study Thom's unoriented cobordism ring. Nath and Sankaran generalized these spaces in \cite{NS19} by introducing the \emph{generalized Dold manifolds} $P(m,\nu)$, which are fiber bundles over the real projective space $\mathbb{R}P^m$ with complex partial flag manifolds $\mathbb{C}G(\nu)$ as fibers .  More precisely, for a partition $\nu=(n_1,\ldots,n_\ell)$ of $n$, the
generalized Dold manifold $P(m,\nu)$ is defined as the orbit space
$$P(m,\nu):=(\mathbb S^m\times\mathbb CG(\nu))/\mathbb Z_2,$$ where
$\mathbb Z_2$ acts diagonally, antipodally on $\mathbb S^m$ and by complex
conjugation on $\mathbb CG(\nu)$. Here,
complex partial flag manifold $\mathbb CG(\nu)$ is the homogeneous space
\[
\mathbb CG(\nu)=U(n)/(U(n_1)\times\cdots\times U(n_\ell)).
\]

Since a generalized Dold manifold $P(m,\nu)$ is fibered by the complex partial flag manifold $\mathbb CG(\nu)$ over $\mathbb RP^m$, understanding the existence or nonexistence of nonzero degree maps between partial flag manifolds serves as a natural starting point, which is provided by \cite{Man26}. We then lift continuous maps between generalized Dold manifolds to their covering spaces of the form $\mathbb S^m\times\mathbb CG(\nu)$. By analyzing the induced homomorphisms on second cohomology through braid hyperplane arrangements, together with the variations of the corresponding heights across different intersections of these hyperplanes, we determine whether the associated rational cohomology algebras admit graded algebra monomorphisms. Whenever such a monomorphism does not exist, this provides an obstruction to the existence of continuous maps of nonzero degree between the corresponding generalized Dold manifolds.

Our main results are summarized in the following theorem.

\begin{thma}\label{main thm intro}
    Let $\nu=(n_1, \ldots n_\ell)$ and $\mu =(m_1, \ldots , m_\wp).$
    Let $f:P(m,\nu)\longrightarrow P(r,\mu)$ be a continuous map between two oriented, same dimensional generalized Dold manifolds, such that one of the following holds:
    \begin{align*}
        &\text{\textit{(1)} $m\neq r$ and $(m,\mu)\neq (2, (1,1))$}&& (Proposition~\ref{degree zero, m r are distinct})\\
        & \text{\textit{(2)} $m=r$ and $(\ell, \wp)\neq (2,2)$,}&& (Theorem~\ref{deg zero, one not grassmannian})\\
        &\text{\textit{(3)} $(m,\ell, \wp)=(r,2,2)$ and $\min\{n_i\}_i< \min \{m_i\}_i.$}& &(Proposition~\ref{deg zero, min ni is less})
    \end{align*}
    Then the Brouwer degree of $f$ is zero.
\end{thma}

As a direct application of these degree obstructions, we address a problem of cohomological rigidity. A class of manifolds is said to be \emph{cohomologically rigid} if the isomorphism type of its cohomology algebra determines its topological type up to a specified equivalence, such as homotopy equivalence, homeomorphism, or diffeomorphism. 
The study of cohomological rigidity has emerged as an active topic in algebraic topology. Several works in this direction can be found; see, for examples, \cite{CMS10, PS16, HK22, Man26}.

Since an isomorphism between rational cohomology algebras induces graded algebra monomorphisms in both directions, the nonexistence of such monomorphisms provides an effective obstruction to cohomology-ring isomorphisms. We use these obstructions to establish cohomological rigidity for the class of product spaces consisting of a spheres and a complex or quaternionic partial flag manifolds.

\begin{thmb}[Theorem~\ref{thm cohomological rigidity}]
    Let $\mathbb{F} \in \{\mathbb{C}, \mathbb{H}\}$. Suppose there exists a graded algebra isomorphism between $H^*(\mathbb{S}^m \times \mathbb{F}G(\nu); \mathbb{Q})$ and $H^*(\mathbb{S}^r \times \mathbb{F}G(\mu); \mathbb{Q})$. Then $m = r$ and $\nu$ is a permutation of $\mu$. In particular, the products are homeomorphic.
\end{thmb}

This article is organized as follows. In Section~\ref{prelim}, we recall the necessary topological foundations, including Brouwer degree, generalized Dold manifolds, lifts to the corresponding covering spaces, and height functions associated with degree-two cohomology classes. In Section~\ref{nonexistance on nonzero deg}, we establish Proposition~\ref{degree zero, m r are distinct}, Theorem~\ref{deg zero, one not grassmannian}, and Proposition~\ref{deg zero, min ni is less}, which provide obstructions to maps of nonzero degree. Finally, Section~\ref{sec cohomological rigidity} applies the resulting constraints on graded algebra isomorphisms to prove Theorem~\ref{thm cohomological rigidity}.

\vspace{.2cm}
\textbf{Declaration on AI use:} Large Language Models were used strictly to refine text readability, correct grammar, and format LaTeX. The author conducted all mathematical derivations, proved all results, and accepts full responsibility for the content of this article.

\section{Preliminaries}\label{prelim}

In this section, we recall the basic notions and results that will be used throughout the paper. We begin with a brief review of topological degree and its relation to homology and cohomology. We then recall the relevant properties of generalized Dold manifolds and their cohomology rings. We also recall the lifting properties of continuous maps between generalized Dold manifolds to the covering product spaces.

\subsection{Brouwer degree}

Let $M$ and $N$ be connected, closed, oriented $n$-dimensional topological manifolds. The orientations determine fundamental classes $[M]\in H_n(M;\mathbb Z)$ and $[N]\in H_n(N;\mathbb Z)$. 
Let $\mu_M\in H^n(M;\mathbb Z)$ and $\mu_N\in H^n(N;\mathbb Z)$ denote the fundamental cohomology classes, normalized by $\langle\mu_M,[M]\rangle=1$ and $\langle\mu_N,[N]\rangle=1$, where $\langle\ ,\ \rangle\colon H^n(M;\mathbb Z)\times H_n(M;\mathbb Z)
\to\mathbb Z$ denote the Kronecker pairing.

For a continuous map $f\colon M\to N$, the \emph{Brouwer degree} of $f$, denoted by $\deg(f)\in\mathbb Z$, is the unique integer satisfying $f_*[M]=\deg(f)[N]$. Equivalently, by the naturality of the Kronecker pairing, one has $f^*(\mu_N)=\deg(f)\mu_M$. The notion of degree of a map is due to Brouwer \cite{Bro11}.

The degree of a map satisfies the following properties:
\begin{itemize}
    \item 
    If $f,g\colon M\to N$ are homotopic, then $\deg(f)=\deg(g)$.

    \item 
    If $f\colon M\to N$ and $g\colon N\to P$ are continuous maps between connected, closed, oriented manifolds of the same dimension, then $\deg(g\circ f)=\deg(g)\deg(f)$.

    \item 
    We have $\deg(\operatorname{id}_M)=1$ if the orientations chosen on both domain and codomain are same, while every nullhomotopic map has degree zero when $\dim M>0$.
\end{itemize}

When $M$ and $N$ are smooth manifolds, the homological definition of degree agrees with the classical differential topological definition. If $f\colon M\to N$ is smooth and $y\in N$ is a regular value of $f$, then
\[
\deg(f):=\sum_{x\in f^{-1}(y)}\deg_x(f),
\]
where $\deg_x(f)=1$ or $-1$ according as $Df_x\colon T_xM\to T_yN$ preserves or reverses the orientations. Moreover, for every top-degree differential form $\omega\in\Omega^n(N)$, one has $\int_M f^*\omega=\deg(f)\int_N\omega$. Thus, the homological definition extends the classical differential-topological notion of degree from smooth maps to arbitrary continuous maps.

When $N=S^n$, the degree completely determines the homotopy class of a map. More precisely, the degree induces an bijection $$\deg\colon[M,S^n]\longrightarrow\mathbb Z, \; [f]\mapsto \deg(f),$$ where $[M,S^n]$ denotes the family of homotopy classes of maps from $M$ to $\mathbb S^m$.   This result is  known as the Hopf degree theorem.

We need the following standard criterion relating degree to rational cohomology.

\begin{lemma}\label{deg zero iff not injective}
Let $f \colon M \to N$ be a continuous map between closed, oriented manifolds of equal dimensions. The degree $\deg(f)$ is nonzero if and only if $f^* \colon H^*(N; \mathbb{Q}) \to H^*(M; \mathbb{Q})$ is an injective ring homomorphism.
\end{lemma}

\subsection{Partial flag manifolds}
Fix an integer partition $\nu = (n_1, \dots, n_\ell)$ of $n = \sum_{i=1}^\ell n_i$. The complex partial flag manifold associated with $\nu$, denoted by $\mathbb{C}G(\nu)$, is defined as the homogeneous space $(n)\big/U(n_1) \times \cdots \times U(n_\ell).$

Equivalently, $\mathbb{C}G(\nu)$ parametrizes all ordered $\ell$-tuples of mutually orthogonal complex subspaces $(V_1, \dots, V_\ell)$ in $\mathbb{C}^n$ satisfying $\dim_{\mathbb{C}} V_k = n_k$ for $1 \le k \le \ell$. Its complex dimension is given by $\dim_{\mathbb{C}} \mathbb{C}G(\nu) = \sum_{1 \le i < j \le \ell} n_i n_j$.

Over $\mathbb{C}G(\nu)$, there exist canonical vector bundles $\gamma_1, \dots, \gamma_\ell$ of complex ranks $n_1, \dots, n_\ell$, respectively, whose fibers at a point $(V_1, \dots, V_\ell)$ are the subspaces $V_1, \dots, V_\ell$. The Whitney sum of these canonical bundles is trivial:
\begin{equation*}
    \gamma_1 \oplus \cdots \oplus \gamma_\ell \cong \varepsilon^n_{\mathbb{C}},
\end{equation*}
where $\varepsilon^n_{\mathbb{C}}$ denotes the trivial bundle $\mathbb{C}G(\nu) \times \mathbb{C}^n$. In the Whitney sum formula, taking the total Chern classes $c(\gamma_j) = 1 + \sum_{i=1}^{n_j} c_i(\gamma_j)$ gives the relation
\begin{equation*}
    \prod_{j=1}^\ell\left( 1 + \sum_{i=1}^{n_j} c_i(\gamma_j) \right) = 1 \in H^*\big(\mathbb{C}G(\nu); \mathbb{Z}\big).
\end{equation*}

The integral cohomology algebra $H^*\big(\mathbb{C}G(\nu); \mathbb{Z}\big)$ admits the presentation
\begin{equation}\label{cohom of flag}
    H^*\big(\mathbb{C}G(\nu); \mathbb{Z}\big) \;\cong\; \mathbb{Z}\big[c_{i,j} : 1 \le j \le \ell, \, 1 \le i \le n_j\big]\Big/\langle h_1, h_2, \dots, h_n \rangle,
\end{equation}
where $c_{i,j}$ corresponds to the $i$-th Chern class $c_i(\gamma_j)$ of degree $2i$, and the ideal generators $h_k$ arise as the homogeneous parts of degree $2k$ in the expansion
\begin{equation*}
    1 + \sum_{k=1}^n h_k = \prod_{j=1}^\ell \left( 1 + c_{1,j} + c_{2,j} + \cdots + c_{n_j,j} \right).
\end{equation*}
For the details of cohomology ring of a complex partial flag manifold, we refer to \cite{Bor53}.

For $\ell=2$, $\nu = (n_1, n_2)$ and the space $\mathbb{C}G(n_1, n_2)$ is the complex Grassmannian of $n_1$-planes in $\mathbb{C}^{n_1+n_2}$. Let $c_i = c_i(\gamma_1)$ and $\bar{c}_j = c_j(\gamma_2)$ represent the Chern classes of the canonical vector bundles, respectively. The relation $c(\gamma_1)c(\gamma_2) = 1$ gives
\begin{equation}\label{eq:grassmann-full}
    H^*\big(\mathbb{C}G(n_1, n_2); \mathbb{Z}\big) \cong \mathbb{Z}[c_1, \dots, c_{n_1}, \bar{c}_1, \dots, \bar{c}_{n_2}]\Big/ \langle h_1, \dots, h_{n_1+n_2} \rangle,
\end{equation}
where $h_k = \sum_{p+q=k} c_p \bar{c}_q$ for $1 \le k \le n_1+n_2$. Recursively expressing each $\bar{c}_j$ in terms of $c_1, \dots, c_{n_1}$ using the relations $h_1 = \dots = h_{n_2} = 0$ allows us to eliminate the variables $\bar{c}_1, \dots, \bar{c}_{n_2}$ entirely. This transforms the presentation \eqref{eq:grassmann-full} into the reduced presentation
\begin{equation}\label{eq:grassmann-reduced}
    H^*\big(\mathbb{C}G(n_1, n_2); \mathbb{Z}\big) \cong \mathbb{Z}[c_1, \dots, c_{n_1}]\Big/\langle h_{n_2+1}, \dots, h_{n_1+n_2} \rangle.
\end{equation}

The above description also has a analogue for quaternions $\mathbb H$. In fact, the integral cohomology ring of a quaternionic partial flag manifold $\mathbb HG(\nu)$ admits a presentation with the same presentation as that of $\mathbb CG(\nu)$, with the degrees of the generators doubled. This degree-doubling is induced by the correspondence between the relevant characteristic classes in the complex and quaternionic settings. Consequently, the defining relations are obtained from those in the complex case by replacing each generator by its quaternionic counterpart and doubling its degree.

For the defining sequence $\nu=(n_1, \ldots, n_\ell)$ of a partial flag manifold, we denote its length, the number $\ell$, by $|\nu|$.

\subsection{Generalized Dold manifolds} \label{gen dold}

In his seminal work \cite{Do56}, Dold introduced the classical \textit{Dold manifolds} 
\[
P(m,n) := (\mathbb{S}^m \times \mathbb{C}P^n) / \!\sim,\quad \text{where }(s, L) \sim (-s, \bar{L}).
\]
Here, $L \mapsto \bar{L}$ denotes the complex conjugation on $\mathbb{C}P^n$ induced from standard conjugation on $\mathbb C^{n+1}$. These spaces were used to construct generators for René Thom's unoriented cobordism ring in odd dimensions. 

This construction is generalized in \cite{NS19, MS22}, and several aspects of this generalizations have been studied in \cite{NS21, MS24, MS26}. In a generalized Dold space, the sphere $\mathbb{S}^m$ is replaced by an arbitrary topological space $S$ equipped with a free involution $\alpha$, while the complex projective space $\mathbb{C}P^n$ is replaced by a topological space $X$ equipped with an involution $\sigma\colon X\to X$ whose fixed-point set is nonempty, i.e., $\operatorname{Fix}(\sigma)\neq\emptyset.$ This leads to the following definition.

\begin{definition}
Let $S$ and $X$ be topological spaces equipped with involutions
$\alpha\colon S\to S$ and $\sigma\colon X\to X$, respectively, where $\alpha$
is free and $\operatorname{Fix}(\sigma)\neq\emptyset$. The quotient space
\begin{equation}\label{gen dold space}
    P(S,\alpha,X,\sigma)
    :=(S\times X)/\!\sim,\quad \text{where }(s,x)\sim(\alpha(s),\sigma(x)),
\end{equation}
is called a \textit{generalized Dold space} (GDS). When the involutions are
understood, we simply write $P(S,X)$. If $P(S,X)$ is a manifold, we call it a \textit{generalized Dold manifold }(GDM).
\end{definition}

The quotient projection $S \times X \to P(S, X)$ defines a two-fold covering map. Let $Y := S/\langle\alpha\rangle$ denote the orbit space of $S$ under the action of $\alpha$. Then $P(S,X)$ carries the structure of a fiber bundle over $Y$ with fiber $X$:
\begin{equation}\label{proj}
    p \colon P(S,X) \twoheadrightarrow Y, \quad [s,x] \mapsto [s].
\end{equation}
For any fixed point $x_0 \in \mathrm{Fix}(\sigma)$, the assignment $[s] \mapsto [s, x_0]$ determines a global section $s \colon Y \hookrightarrow P(S,X)$. More generally, this yields an embedding 
\[
Y \times \mathrm{Fix}(\sigma) \hookrightarrow P(S,X),
\]
where $\mathrm{Fix}(\sigma) \subseteq X$ carries the subspace topology.

In the literature, a frequently studied class of GDSs is obtained by keeping the sphere intact and replacing the complex projective space with a complex partial flag manifold equipped with the standard complex conjugation. These generalized Dold manifolds $P(\mathbb S^m,\mathbb CG(\nu))$ are simply denoted by $P(m,\nu).$

\subsection{Rational cohomology of GDMs $P(m,\nu)$} 

In this article, we are primarily interested in the generalized Dold manifolds $P(m,\nu)$ fibred by complex partial flag manifolds $\mathbb CG(\nu)$ over real projective spaces $\mathbb RP^m$.

By the K\"unneth formula, the rational cohomology of $\mathbb{S}^m \times \mathbb{C}G(\nu)$ is given by
\begin{equation}\label{Cohomology of H_times}
    H^*(\mathbb{S}^m \times \mathbb{C}G(\nu)); \mathbb{Q}) \cong H^*(\mathbb{S}^m; \mathbb{Q}) \otimes H^*(\mathbb{C}G(\nu); \mathbb{Q}) \cong H^*(\mathbb{C}G(\nu); \mathbb{Q})[u]\big/\langle u^2 \rangle,
\end{equation}
where $u \in H^m(\mathbb{S}^m; \mathbb{Q})$ denotes the fundamental cohomology class of $\mathbb{S}^m$. As a $\mathbb{Q}$-module, the cohomology algebra of the product admits the decomposition
\begin{equation}\label{H in terms of u}
    H^*(\mathbb{S}^m \times \mathbb{C}G(\nu)); \mathbb{Q})\cong H^*(\mathbb{C}G(\nu); \mathbb{Q}) \oplus u H^*(\mathbb{C}G(\nu); \mathbb{Q}),
\end{equation}
where $H^*(\mathbb{C}G(\nu); \mathbb{Q})$ is viewed as a subring of $H^*(\mathbb{S}^m \times \mathbb{C}G(\nu)); \mathbb{Q})$.

The product involution $\theta := \alpha \times \sigma$ on $\mathbb{S}^m \times \mathbb{C}G(\nu)$ induces an automorphism $\theta^*$ on $H^*(\mathbb{S}^m \times \mathbb{C}G(\nu)); \mathbb{Q})$ defined on generators by
\begin{equation}\label{defn of theta*}
    \theta^*(c_i) = (-1)^i c_i, \quad 1 \le i \le k, \qquad \text{and} \qquad
    \theta^*(u) =(-1)^{m+1}u.
\end{equation}

The rational cohomology ring $H^*(P(m,\nu); \mathbb{Q})$ was determined in \cite{MS24} by identifying it with the fixed point subring under $\theta^*$.

\begin{theorem}[{\cite[Theorem 3.13]{MS24}}]\label{cohomology of P(m,nu)}
    The cohomology algebra $H^*(P(m,\nu); \mathbb Q)$, where $\nu=(n_1,\ldots, n_\ell)$, is isomorphic to the invariant subalgebra $\mathrm{Fix}(\theta^*) \subseteq H^*(\mathbb{S}^m \times \mathbb{C}G(\nu); \mathbb Q)$, which is generated by the following monomials in terms of Chern classes:

\begin{enumerate}[\quad]
    \item \textbf{When $m$ is even.}
    \begin{align*}
        &u\, c_{2p-1}(\gamma_i), && 1 \le 2p-1 \le n_i, \; 1 \le i \le \ell; \\
        &c_{2j}(\gamma_i), && 1 \le 2j \le n_i, \; 1 \le i \le \ell; \quad \text{and} \\
        &c_{2p-1}(\gamma_i) \, c_{2q-1}(\gamma_j), && 1 \le 2p-1 \le n_i, \; 1 \le 2q-1 \le n_j, \; 1 \le i \le j \le \ell.
    \end{align*}

    \item \textbf{When $m$ is odd.}
    \begin{align*}
        &u; \\
        &c_{2j}(\gamma_i), && 1 \le 2j \le n_i, \; 1 \le i \le \ell; \quad \text{and} \\
        &c_{2p-1}(\gamma_i) \, c_{2q-1}(\gamma_j), && 1 \le 2p-1 \le n_i, \; 1 \le 2q-1 \le n_j, \; 1 \le i \le j \le \ell.
    \end{align*}
\end{enumerate}
\end{theorem}

An explicit presentation of $H^*(P(m,\nu); \mathbb{Q})$ as a quotient of a polynomial algebra follows as a special case of \cite[Theorem 3.14]{MS24}.

\begin{remark}
We recall the orientability criterion for \(P(m,\nu)\) from
Remark~3.7(ii) of \cite{MS24}. Let $D=\dim_{\mathbb C}\mathbb CG(\nu)$, and let $\omega$ be a generator of $H^{2D}(\mathbb CG(\nu);\mathbb Q)$. Also, let $u$ be a generator of $H^m(\mathbb S^m;\mathbb Q)$. Then $u\omega$ generates $H^{m+2D}(\mathbb S^m\times\mathbb CG(\nu);\mathbb Q)\cong\mathbb Q$.

Since
$\alpha^*(u)=(-1)^{m+1}u$ and
$\sigma^*(\omega)=(-1)^{D}\omega$, we have
\[
\theta^*(u\omega)
=(-1)^{m+D+1}u\omega.
\]
Thus $u\omega$ is $\theta^*$-invariant if and only if $m+D$ is odd. By Theorem~\ref{cohomology of P(m,nu)}, this is equivalent to $H^{m+2D}(P(m,\nu);\mathbb Q)\cong\mathbb Q$. Hence
\begin{equation}\label{orientability of GDM}
    P(m,\nu)\ \text{is orientable}
\quad\Longleftrightarrow\quad
m+\dim_{\mathbb C}\mathbb CG(\nu)\ \text{is odd}.
\end{equation}

\end{remark}

For details of the preliminaries of this article, we refer to \cite{Spa66, MS74, Ful98}.

\subsection{Lifts between covering spaces}
Let us consider two generalized Dold manifolds $P(m,\nu)$ and $P(r,\mu)$. Let $p_{m,\nu}\colon \mathbb S^m\times\mathbb CG(\nu)\to P(m,\nu)$ and $p_{r,\mu}\colon \mathbb S^r\times\mathbb CG(\mu)\to P(r,\mu)$ be the double covering maps. One can show that every continuous map $f\colon P(m,\nu)\to P(r,\mu)$ admits a lift $\widetilde f\colon \mathbb S^m\times\mathbb CG(\nu)\to \mathbb S^r\times\mathbb CG(\mu)$
such that $p_{r,\mu}\circ\widetilde f=f\circ p_{m,\nu}$.

Recall that $P(m,\nu)$ is a $\mathbb CG(\nu)$-bundle over $\mathbb RP^m$:
\[
    \mathbb CG(\nu)\xhookrightarrow{\quad} P(m,\nu)
    \longrightarrow\mathbb RP^m.
\]
Since every complex partial flag manifold is simply connected, the homotopy exact sequence of this fibre bundles give
$\pi_1(P(m,\nu))\cong\pi_1(\mathbb RP^m)$ and $\pi_1(P(r,\mu))\cong \pi_1(\mathbb RP^r)$. Thus
\[
    \pi_1(P(m,\nu))\cong
    \begin{cases}
    \mathbb Z, & m=1,\\
    \mathbb Z_2, & m>1,
    \end{cases}
    \quad\text{and}\quad \pi_1(P(r,\mu))\cong
    \begin{cases}
    \mathbb Z, & r=1,\\
    \mathbb Z_2, & r>1.
    \end{cases}
\]
Since $\pi_1(\mathbb S^m\times\mathbb CG(\nu))\cong\pi_1(\mathbb S^m)$. Moreover, the double covering $p_{m,\nu}$ induces
\[
    (p_{m,\nu})_*
    \bigl(\pi_1(\mathbb S^m\times\mathbb CG(\nu))\bigr)
    =
    \begin{cases}
    2\mathbb Z\subset\mathbb Z, & m=1,\\
    0\subset\mathbb Z_2, & m>1.
    \end{cases}
\]
Likewise,
\[
    (p_{r,\mu})_*
    \bigl(\pi_1(\mathbb S^r\times\mathbb CG(\mu))\bigr)
    =
    \begin{cases}
    2\mathbb Z\subset\mathbb Z, & r=1,\\
    0\subset\mathbb Z_2, & r>1.
    \end{cases}
\] 
By the lifting criterion, it suffices to prove
\[
    (f\circ p_{m,\nu})_*
    \bigl(\pi_1(\mathbb S^m\times\mathbb CG(\nu))\bigr)
    \subseteq
    (p_{r,\mu})_*
    \bigl(\pi_1(\mathbb S^r\times\mathbb CG(\mu))\bigr).
\]

If $m>1$, then $\pi_1(\mathbb S^m\times\mathbb CG(\nu))=0$, so the required inclusion is immediate. 

If $m=1$, then $\pi_1(\mathbb S^1\times\mathbb CG(\nu))\cong\mathbb Z$, and $(p_{1,\nu})_*(\mathbb Z)=2\mathbb Z$. Hence
$(f\circ p_{1,\nu})_*(\mathbb Z)=f_*(2\mathbb Z).$ Now observe the two cases: $r>1$ and $r=1.$
If $r>1$, then $\pi_1(P(r,\mu))=\mathbb Z_2$, and therefore $f_*(2\mathbb Z)=0=(p_{r,\mu})_* \bigl(\pi_1(\mathbb S^r\times\mathbb CG(\mu))\bigr).$ If $r=1$, then $f_*\colon\mathbb Z\to\mathbb Z$ is a group homomorphism, so $f_*(2\mathbb Z) \subseteq 2\mathbb Z =(p_{1,\mu})_* \bigl(\pi_1(\mathbb S^1\times\mathbb CG(\mu))\bigr).$

Thus the lifting criterion is satisfied in all cases. Consequently, there exists a continuous map $\widetilde f\colon \mathbb S^m\times\mathbb CG(\nu)\to \mathbb S^r\times\mathbb CG(\mu)$ such that $p_{r,\mu}\circ\widetilde f=f\circ p_{m,\nu}.$ We summarize the observation in the following remark.
\begin{remark}\label{covering lift}
    For any continuous map $f:P(m,\nu)\to P(r,\mu)$, there exists a lift $\widetilde{f}:\mathbb S^m\times\mathbb CG(\nu)\to\mathbb S^r\times\mathbb CG(\mu)$ making the diagram commute; that is, $p_{r,\mu}\circ\widetilde{f}=f\circ p_{m,\nu}$, where $p_{m,\nu}$ and $p_{r,\mu}$ denote the respective covering projections.
\end{remark}

\section{Nonexistence of nonzero degree maps}\label{nonexistance on nonzero deg}

In this section, we study the degrees of maps between generalized Dold manifolds $P(m,\nu)$. Recall that $P(m,\nu)$ is fibred by complex partial flag manifold $\mathbb CG(\nu)$ over real projective space $\mathbb RP^n$. Our approach builds on the study of the nonexistence of nonzero-degree maps between complex partial flag manifolds carried out in \cite{Man26}. A detailed understanding of the fibers $\mathbb{C}G(\nu)$ in this context allows us to derive certain conclusions for the total spaces $P(m,\nu)$.

In view of Lemma~\ref{deg zero iff not injective}, the problem of determining whether a map can have nonzero degree reduces to studying the existence of algebra monomorphisms between the corresponding rational cohomology algebras. For complex partial flag manifolds, there are several instances in which homomorphisms between cohomology algebras are determined by their behavior on degree-two cohomology classes, together with the heights of these classes; see, for example, \cite{ GH78, HH84, MS26, Man26}. Hence, Remark~\ref{covering lift} motivates to begin our study by analyzing the heights of degree-two cohomology classes in the products of spheres and complex partial flag manifolds. This will provide
some necessary understanding for studying the degrees of maps between generalized Dold manifolds.

\subsection{Heights of degree two cohomology classes}

Let $x_1,\ldots,x_\ell$ denote the first Chern classes of the canonical vector bundles of rank $n_i$ over $\mathbb CG(\nu)$, where $\nu=(n_1, \ldots , n_\ell)$. Recall the cohomology algebra of $\mathbb CG(\nu)$ from \eqref{cohom of flag} and note that $x_1,\ldots, x_\ell$ form a $\mathbb Z$-basis of the $\mathbb Z$-module $H^2(\mathbb CG(\nu);\mathbb Z)$. 

To study the heights of degree-two classes in $H^*(\mathbb  S^m \times \mathbb CG(\nu);\mathbb Z)$, we use the following result from \cite{BHH83} on heights of degree-two classes of complex partial flag manifolds.

\begin{theorem}[{\cite[Theorem 3.1]{BHH83}}]\label{heights of deg 2 elements}
    Let $x = \sum_{i=1}^{\ell} a_i x_i \in H^2(\mathbb{C}G(\nu); \mathbb{Z})$ and let $\{b_1, b_2, \ldots, b_t\}$ be the set of distinct values in $\{a_1, a_2, \ldots, a_\ell\}$. 
    For each $1 \le j \le t$, set
    \[
    m_j = \sum_{\{i\; : \; a_i = b_j\}} n_i.
    \]
    Then the height of $x$, denoted $h(x)$, is given by
    \[
    h(x) = \sum_{1 \le i < j \le t} m_i m_j.
    \]
\end{theorem}

\begin{corollary}\label{kahler class descripition}
    By Theorem \ref{heights of deg 2 elements}, a degree-two cohomology class $x=\sum_{i=1}^{\ell}a_i x_i$ has maximal height, equal to the complex dimension of $\mathbb{C}G(\nu)$, namely
    $\sum_{1\leq i<j\leq \ell}n_i n_j$, if and only if the coefficients $a_1,\ldots,a_\ell$ are pairwise distinct. Such degree-two cohomology classes are precisely the K\"ahler classes in cohomology for $\mathbb{C}G(\nu)$.
\end{corollary}

We now determine the heights of degree-two  cohomology classes in $H^*(\mathbb S^m\times\mathbb CG(\nu);\mathbb Z)$.

If $m\neq2$, then $H^2(\mathbb S^m;\mathbb Z)=0$, and hence the K\"unneth theorem gives
\[
    H^2(\mathbb S^m\times\mathbb CG(\nu);\mathbb Z)
    \cong H^2(\mathbb CG(\nu);\mathbb Z).
\]
Thus every degree-two cohomology class has the form $z=\sum_{i=1}^{\ell}a_i x_i.$ Since the second projection map $p_2: \mathbb S^m \times \mathbb CG(\nu)\to \mathbb CG(\nu)$ induces monomorphism in cohomology, the height of $z$ is given by Theorem \ref{heights of deg 2 elements}.
Therefore, when $m\neq2$, taking the product with the sphere $\mathbb S^m$ does not change the height function on degree-two cohomology.

Suppose now that $m=2$. Let $u$ be a generator of $H^2(\mathbb S^2;\mathbb Z)$, and denote $p_1^*(u)$ by the same symbol $u$, where $p_1\colon\mathbb S^2\times\mathbb CG(\nu)\to\mathbb S^2$ is the first projection. Then
\[
    H^2(\mathbb S^2\times\mathbb CG(\nu);\mathbb Z)
    \cong
    \mathbb Z u\oplus H^2(\mathbb CG(\nu);\mathbb Z).
\]
Hence every degree-two cohomology class $z\in H^2(\mathbb S^2\times \mathbb CG(\nu);\mathbb Z)$ has a unique expression
\[
z=au+x,
\qquad
a\in\mathbb Z,\quad
x\in H^2(\mathbb CG(\nu);\mathbb Z).
\]

Suppose first that $x\neq0$, and put $h=h(x)$. Now we have
\[
  z^{h+1}=(d+1)aux^h\neq 0,
\]
since  $u^2=0$ and $ux^h\neq0$ in $H^*(\mathbb S^m \times\mathbb CG(\nu);\mathbb Z)$. Consequently, if $a\neq0$, then $z^{h+1}\neq0$. On the other hand,
\[
z^{h+2}
=(h+2)aux^{h+1}=0.
\]
It follows that $h(au+x)=h(x)+1$ if $a\neq 0, x\neq 0$.  Thus, for $m=2$, we have
\begin{equation}\label{heights for deg two for products}
     h(au+x)=
    \begin{cases}
    h(x),& a=0,\ x\neq0,\\
    h(x)+1,& a\neq0,\ x\neq0,\\
    1,& a\neq0,\ x=0.
    \end{cases}
\end{equation}

Hence a nonzero component in the direction of $u$ increases the height by one, whereas for $m\neq2$ the height is exactly coming from the partial flag manifold.

\subsection{Degrees of maps from $P(m,\mu)$ to $P(r, \nu)$, when $m\neq r$} Now we focus our attention to studying the existence of nonzero degree maps between generalized Dold manifolds $P(m,\nu)$ and $P(r, \mu)$ for the case when the sphere dimensions $m $ and $r$ are different. First we prove the following lemma.

\begin{lemma}\label{no mono m r distinct}
    Let $\phi:H^*(\mathbb S^m\times \mathbb FG(\nu);\mathbb Q)\longrightarrow H^*(\mathbb S^r\times \mathbb FG(\mu);\mathbb Q)$ be a graded algebra homomorphism such that $m\neq r$ and $m+\dim_{\mathbb R}\mathbb FG(\nu)= r+\dim_{\mathbb R}\mathbb FG(\mu)$. Further assume that $(m, \mu)\neq (2, (1,1)).$ Then $\phi$ is not a monomorphism.
\end{lemma}
\begin{proof}
    We only prove the result for $\mathbb{F}=\mathbb{C}$. Since the integral cohomology $H^*(\mathbb{H}G(n_1,n_2);\mathbb{Z})$ admits the same presentation as $H^*(\mathbb{C}G(n_1,n_2);\mathbb{Z})$, up to a degree-doubling isomorphism, the case $\mathbb{F}=\mathbb{H}$ follows by a similar argument.

    Let $N_1 = \dim_{\mathbb{C}}(\mathbb{C}G(\nu))$ and $N_2 = \dim_{\mathbb{C}}(\mathbb{C}G(\mu))$. The real dimensions of the generalized Dold manifolds are given by
    \[
    \dim_{\mathbb{R}} P(m, \nu) = m + 2N_1 \quad \text{and} \quad \dim_{\mathbb{R}} P(r, \mu) = r + 2N_2.
    \]
    Recall that from the observation \eqref{orientability of GDM} a generalized Dold manifold $P(m, \nu)$ is orientable if and only if $m + N_1$ is odd. Since both $P(m,\nu)$ and $P(r,\mu)$ are assumed to be oriented and have equal dimension ($m + 2N_1 = r + 2N_2$), the integers $N_1$ and $N_2$ must have the same , using \eqref{orientability of GDM}. Moreover, since $m \neq r$, we have $2(N_1 - N_2) = r - m \neq 0$, which implies $|N_1 - N_2| \ge 2$. We shall proceed with two cases: (i) $N_2-N_1\ge2$ and (ii) $N_1-N_2\ge 2$.

    Let $\{y_1, \ldots, y_\wp\}$ denote the basis for $H^2(\mathbb{C}G(\mu); \mathbb{Q})$, where $y_i$ is the the first Chern class of the canonical vector bundle over $\mathbb CG(\mu)$ of rank $m_i$. Choose a K\"ahler class $y \in H^2(\mathbb{C}G(\mu); \mathbb{Q})$ defined as a linear combination of $y_1, \ldots, y_\wp$ with distinct coefficients. By Corollary~\ref{kahler class descripition}, the top power $y^{N_2} \in H^{2N_2}(\mathbb{C}G(\mu); \mathbb{Q})$ is nonzero, whereas $y^{k} = 0$ for all $k > N_2$.
    The image of $y$ under $\phi$ is expressed as
    \[
    \phi(y) = a u_m + b \cdot z, 
    \]
    where $a, b \in \mathbb{Q}$, $u_m \in H^m(\mathbb{S}^m; \mathbb{Q})$ is the fundamental cohomology class, and $z \in H^2(\mathbb{C}G(\nu); \mathbb{Q})$. Note that $a = 0$ whenever $m \neq 2$.

    \medskip
    \noindent\textbf{Case 1:} $N_2 - N_1 \ge 2$. 
    Applying $\phi$ to $y^{N_2} \neq 0$ and using $u_m^2 = 0$, we obtain
    \[
    \phi(y^{N_2}) = N_2 a b^{N_2-1} z^{N_2-1} u_m + b^{N_2} z^{N_2}.
    \]
    Since $N_2 \ge N_1 + 2 > \dim_{\mathbb{C}}(\mathbb{C}G(\nu))$, both $z^{N_2-1}$ and $z^{N_2}$ vanish in $H^*(\mathbb{C}G(\nu); \mathbb{Q})$ by degree considerations. Hence, $\phi(y^{N_2}) = 0$, which shows that $\ker(\phi) \neq \{0\}$, so $\phi$ is not a monomorphism.

    \medskip
    \noindent\textbf{Case 2:} $N_1 - N_2 \ge 2$. 
    Since $y^{N_1} = 0$ in $H^*(\mathbb{C}G(\mu); \mathbb{Q})$, we have
    \[
    0 = \phi(y^{N_1}) = N_1 a b^{N_1-1} z^{N_1-1} u_m + b^{N_1} z^{N_1}.
    \]
    This implies that $b = 0$ and $ a b= 0$.

    If $a = 0$, $\phi(y) = 0$ and so $\phi$ is not a monomorphism.

    If $a \neq 0$, $\phi(y) = au_2$, since  $m=2$. Now we have two cases: (i) $y^2 \neq 0$ and (ii) $y^2 = 0$.
    In case (i), $0 \neq y^2 \mapsto \phi(y^2) = a^2 u_2^2 = 0$ implies $\phi$ is not a monomorphism. In case (ii), $y^2 = 0$ implies $\mathbb{C}G(\mu) = \mathbb CP^1$; and consequently, $(2, \mu) = (2,(1,1))$ which is not possible because of hypothesis (ii).

 Therefore, $\phi$ cannot be  a monomorphism.
\end{proof}

As an immediate consequence, we have the following proposition.

\begin{proposition}\label{degree zero, m r are distinct}
    Let $m,r\in\mathbb N$ and $\nu, \mu$ be two finite sequences of natural numbers such that $\dim (\mathbb S^m\times \mathbb CG(\nu))=\dim (\mathbb S^r\times \mathbb CG(\mu))$ and $m\neq r.$ Assume that $(m,\mu)\neq(2,(1,1)).$
    Then for every continuous map $f: S^m\times \mathbb CG(\nu) \longrightarrow S^r\times \mathbb CG(\mu)$, the degree of $f$ is zero. Further, any continuous map $g: P(m,\nu)\longrightarrow P(r, \mu)$ has degree zero,  assuming that $P(m,\nu)$ and $P(r,\mu)$ are orientable.
\end{proposition}

\begin{proof}
   By Remark~\ref{covering lift}, there exists a lift $\widetilde{g}\colon \mathbb{S}^m\times\mathbb{C}G(\nu) \longrightarrow \mathbb{S}^r\times\mathbb{C}G(\mu)$ such that $p_{r,\mu}\circ\widetilde{g}=g\circ p_{m,\nu}$. Since $\deg(p_{m,\nu})=\deg(p_{r,\mu})=2$, it follows that
    $\deg(g)=\deg(\widetilde{g})$. Therefore, it is enough to show that every continuous map $f:\mathbb S^m\times\mathbb FG(\nu) \longrightarrow \mathbb S^r\times\mathbb FG(\mu)$, where $\mathbb F\in\{ \mathbb C, \mathbb H\}$, has degree zero.
     Now using Lemma~\ref{no mono m r distinct}, $f^*$ cannot be a monomorphism. Therefore,  $\deg(f)$ is zero.  This completes the proof.
\end{proof}

Proposition~\ref{degree zero, m r are distinct} extends \cite[Lemma~7.3.3]{Man24}, which established the analogous result for generalized Dold manifolds fibered by complex Grassmannians.

\begin{remark}\label{rem:hypothesis_ii_discussion}
    The exceptional hypothesis {\rm (ii)} in Lemma~\ref{degree zero, m r are distinct} may not be strictly necessary. Specifically, if $(m, \mu = (2, (1,1))$, the  partial flag manifold involved in the codomain has complex dimension $ 1$, which corresponds uniquely to $\mathbb{C}G(\mu) = \mathbb{C}P^1 \cong \mathbb{S}^2$.
    Equating the real dimensions $\dim_{\mathbb{R}} P(2, \nu) $ and $\dim_{\mathbb{R}} P(r, (1,1))$ forces $r = 2\cdot\sum_{1\le i<j\le \ell} n_i n_j $. In this scenario, we consider maps of the form
    \[
    f \colon P\big(2, \nu\big) \longrightarrow P\big(D(\nu), (1,1)\big), \quad \text{where }D(\nu)=\sum_{1\le i<j\le \ell} 2n_i n_j.
    \]
    
    Topologically, one can construct continuous maps of degree $1$ from the product $\mathbb{S}^2 \times \mathbb{C}G(\nu)$ to $\mathbb{S}^{D(\nu)} \times \mathbb{S}^2$ by taking the identity map on the $\mathbb{S}^2$ factor and pairing it with the degree-one quotient map $\mathbb{C}G(\nu) \twoheadrightarrow \mathbb{S}^{D(\nu)}$ that collapses the $(D(\nu) - 1)$-skeleton of a standard CW-structure on $\mathbb{C}G(\nu)$ to a point. However, it requires further exploration to answer the question whether any such map can be chosen to be equivariant with respect to the involutions $\theta = \alpha \times \sigma$ on the product spaces, and thus descend to a continuous map from $P(2, \nu)$ to $P(D(\nu), (1,1))$ of nonzero Brouwer degree.
\end{remark}

\subsection{Degrees of maps from $P(m,\mu)$ to $P(m, \nu)$, when $(|\nu|, |\mu|)\neq (2,2)$}
We now study the existence of nonzero degree maps between two  generalized Dold manifolds fibred by distinct complex partial flag manifolds over the same real projective space.

First, let us recall the following lemmas from \cite{Man26}.

\begin{lemma}[\cite{Man26}, Lemma 3.5, Remark 3.7]\label{T(e_i)}
Let $T:\mathbb Q^m \to \mathbb Q^n$ be an injective linear map with $ m\ge 3$. Define the hypersurfaces $H_{p,q}:=\{(x_1,\dots,x_m)\in\mathbb Q^m : x_p=x_q\},$ for $1\le p<q\le m$, and $K_{u,v}:=\{(y_1,\dots,y_n)\in\mathbb Q^n : y_u=y_v\},$ for $1\le u<v\le n$.
Suppose
\[
T\Bigl(\mathbb Q^m\setminus \bigcup_{1\le p<q\le m} H_{p,q}\Bigr)
\;\subseteq\;
\mathbb Q^n\setminus \bigcup_{1\le u<v\le n} K_{u,v}.
\]
Then $m=n$ and there exist a permutation $\sigma$ on the set $\{1,2,\ldots,m\}$, a nonzero scalar $\lambda$, and scalars $b_i,\; i=1,2,\ldots,m,$ such that
\[
T(e_i)=b_i\mathbf 1_n-\lambda e_{\sigma(i)},\quad \text{for all }i=1,2,\ldots,m,
\]
where $e_i$ denotes the $i$-th standard basis vector and $\mathbf 1_n=\sum_{i=1}^n e_i=(1,\ldots,1)$.
\end{lemma}

\begin{lemma}[\cite{Man26}, Lemma 3.6]\label{extension}
Let $\mathbf 1_k=(1,\ldots,1)=\sum_{i=1}^k e_i\in\mathbb Q^k$, and let
\[
T:\mathbb Q^m/\mathbb Q\mathbf 1_m\to
\mathbb Q^n/\mathbb Q\mathbf 1_n
\]
be an injective linear map. Then there exists an injective linear map $\widetilde T:\mathbb Q^m\to\mathbb Q^n$ such that
$\widetilde T(\mathbb Q\mathbf 1_m)
\subseteq \mathbb Q\mathbf 1_n$ and the induced map on the quotients is $T$.
\end{lemma}

Let us  fix some notations first.
\begin{notation}\label{notation of hyperplanes}
    For two sequences $\nu = (n_1, \ldots, n_\ell)$ and $\mu = (m_1, \ldots, m_\wp)$, let us first fix the notations of the rational cohomology rings by 
    $$\mathcal{R}=\oplus\mathcal R^i = H^*(\mathbb{C}G(\nu); \mathbb{Q}) \quad\text{and}\quad \mathcal{S} = \oplus \mathcal S^i=H^*(\mathbb{C}G(\mu); \mathbb{Q}),$$
    where $\mathcal R^i=H^i(\mathbb CG(\nu);\mathbb Q)$ and $\mathcal S^i=H^i(\mathbb CG(\mu);\mathbb Q)$. Denote $D = \dim_{\mathbb{C}} \mathbb{C}G(\nu) = \dim_{\mathbb{C}} \mathbb{C}G(\mu).$ 
Let $x_i = c_1(\gamma_i)$ and $y_j = c_1(\eta_j)$ be the first Chern classes of the canonical vector bundles over $\mathbb{C}G(\nu)$ and $\mathbb{C}G(\mu)$, respectively. The second cohomology $\mathbb Q$-vector spaces are given by
\[
\mathcal{R}^2 \;\cong\; \bigoplus_{i=1}^\ell \mathbb{Q}x_i \Big/ \mathbb{Q}\sum_{i=1}^\ell x_i \quad\text{and}\quad \mathcal{S}^2 \;\cong\; \bigoplus_{j=1}^\wp \mathbb{Q}y_j \Big/ \mathbb{Q}\sum_{j=1}^\wp y_j.
\]
By the Künneth formula, the second cohomology groups of the product spaces are
\[
\widetilde{\mathcal{R}^2} := H^2(\mathbb{S}^m \times \mathbb{C}G(\nu); \mathbb{Q}) \cong 
\begin{cases}
\mathcal{R}^2, & \text{if } m \neq 2, \\
\mathcal{R}^2 \oplus \mathbb{Q}u, & \text{if } m = 2,
\end{cases}
\]
and
\[
\widetilde{\mathcal{S}^2} := H^2(\mathbb{S}^m \times \mathbb{C}G(\mu); \mathbb{Q}) \cong 
\begin{cases}
\mathcal{S}^2, & \text{if } m \neq 2, \\
\mathcal{S}^2 \oplus \mathbb{Q}u, & \text{if } m = 2,
\end{cases}
\]
where $u \in H^2(\mathbb{S}^2; \mathbb{Q})$ denotes the fundamental cohomology class. For distinct indices $p < q$, we define the hyperplanes in the braid arrangements of the vector spaces $\mathcal{R}^2$ and $\mathcal{S}^2$, respectively, as:
\[
\mathcal{R}^2_{p,q} := \left\{ \sum_{i=1}^\ell a_i x_i \in \mathcal{R}^2 \;\mid\; a_p = a_q \right\} \quad\text{and}\quad \mathcal{S}^2_{p,q} := \left\{ \sum_{j=1}^\wp b_j y_j \in \mathcal{S}^2 \;\mid\; b_p = b_q \right\}.
\]
In addition, we define the corresponding hyperplanes in $\widetilde{\mathcal{R}^2}$ and $\widetilde{\mathcal{S}^2}$ by:
\[
\widetilde{\mathcal{R}}^2_{p,q} := \mathcal{R}^2_{p,q} \oplus \mathbb{Q}u \quad\text{and}\quad \widetilde{\mathcal{S}}^2_{p,q} := \mathcal{S}^2_{p,q} \oplus \mathbb{Q}u\quad \text{respectively}.
\]
Similar notation and constructions apply to the quaternionic case after doubling the degrees. In this setting, the focus will be on the fourth cohomology group rather than the second cohomology group.
\qed
\end{notation}

We need the following lemmas.

\begin{lemma}\label{tripple intersection} We adopt the notation introduced in Notation~\ref{notation of hyperplanes}.
Let $m=2$ and assume $\ell = \wp = n \ge 3$. Let 
\[
\mathcal{F}_{\mathcal{S}} := \{\widetilde{\mathcal{S}}^2_{p,q} \mid 1 \le p < q \le n\} \cup \{\mathcal{S}^2\}\text{ and }\mathcal{F}_{\mathcal{R}}: = \{\widetilde{\mathcal{R}}^2_{p,q} \mid 1 \le p < q \le n\} \cup \{\mathcal{R}^2\}.
\]
Suppose $\phi: \widetilde{\mathcal{S}^2} \to \widetilde{\mathcal{R}^2}$ is a linear isomorphism  such that $\phi (\mathcal F_{\mathcal S})=\mathcal F_{\mathcal R}$. Then $\phi(\mathcal{S}^2) = \mathcal{R}^2$.
\end{lemma}

\begin{proof}
Let us call codimension-$2$ subspace $F \subset \widetilde{\mathcal{R}^2}$ a \emph{triple intersection} of $\mathcal{F}_{\mathcal{R}}$ if there exist three distinct hyperplanes $F_1, F_2, F_3 \in \mathcal{F}_{\mathcal{R}}$ such that $F = F_1 \cap F_2 = F_1 \cap F_3 = F_2 \cap F_3$. We say that a hyperplane $H \in \mathcal{F}_{\mathcal{R}}$ is \emph{a part of triple intersection} if $H = F_i$ for some triple intersection $F$. Analogous definitions apply to $\mathcal{F}_{\mathcal{S}}$ as well.

Since $\phi: \widetilde{\mathcal{S}^2} \to \widetilde{\mathcal{R}^2}$ is a linear isomorphism permuting the collections of hyperplanes, it preserves intersections. In fact, for a triple intersection $F = F_1 \cap F_2 = F_1 \cap F_3 = F_2 \cap F_3$, the image $\phi(F)$ is also a triple intersection, since $$\phi(F) = \phi(F_1) \cap \phi(F_2) = \phi(F_1) \cap \phi(F_3) = \phi(F_2) \cap \phi(F_3) $$ and $\phi(F_1),\phi(F_2), \phi(F_3)$ are distinct hyperplanes as $\phi$ is a linear isomorphism.
Thus, $H \in \mathcal{F}_{\mathcal{S}}$ is a part of triple intersection of $\mathcal{F}_{\mathcal{S}}$ if and only if $\phi(H) \in \mathcal{F}_{\mathcal{R}}$ is a part of triple intersection of $\mathcal{F}_{\mathcal{R}}$.

Now observe that every hyperplane $\widetilde{\mathcal{S}}^2_{p,q}$ is a part of a triple intersection of $\mathcal{F}_{\mathcal{S}}$.
Since $n \ge 3$, for any $1 \le p < q \le n$, we can choose a third index $r \notin \{p, q\}$. Consider the distinct hyperplanes $\widetilde{\mathcal{S}}^2_{p,q}, \widetilde{\mathcal{S}}^2_{p,r}, \widetilde{\mathcal{S}}^2_{q,r} \in \mathcal{F}_{\mathcal{S}}$. Their pairwise intersections are given by
$$\widetilde{\mathcal{S}}^2_{p,q} \cap \widetilde{\mathcal{S}}^2_{p,r}= \widetilde{\mathcal{S}}^2_{p,q} \cap \widetilde{\mathcal{S}}^2_{q,r}=\widetilde{\mathcal{S}}^2_{p,r} \cap \widetilde{\mathcal{S}}^2_{q,r} = \Big\{\sum_{j=1}^n b_j y_j + c u \in \widetilde{\mathcal{S}^2} \;\mid\; b_p = b_q = b_r\Big\}.$$
Hence, $\widetilde{\mathcal{S}}^2_{p,q}$ is a part of a triple intersection.

Now note that the hyperplane $\mathcal{S}^2$ is not a part of any triple intersection of $\mathcal{F}_{\mathcal{S}}$.
Suppose for contradiction that $\mathcal{S}^2$ is a part of a triple intersection. Then there exist \textit{distinct} hyperplanes $\widetilde{\mathcal{S}}^2_{p,q}, \widetilde{\mathcal{S}}^2_{r,s} \in \mathcal{F}_{\mathcal{S}}$ such that

\begin{equation}\label{intersections}
    \mathcal{S}^2 \cap \widetilde{\mathcal{S}}^2_{p,q} = \mathcal{S}^2 \cap \widetilde{\mathcal{S}}^2_{r,s} = \widetilde{\mathcal{S}}^2_{p,q} \cap \widetilde{\mathcal{S}}^2_{r,s}.
\end{equation}
Since $\mathcal{S}^2 \cap \widetilde{\mathcal{S}}^2_{p,q}=\mathcal S_{p,q}^2$ and $\mathcal{S}^2 \cap \widetilde{\mathcal{S}}^2_{r,s}=\mathcal S^2_{r,s}$, \eqref{intersections} implies that $\{p,q\}=\{r,s\}.$ 
 This contradicts the assumption that $\widetilde{\mathcal{S}}^2_{p,q}$ and $\widetilde{\mathcal{S}}^2_{r,s}$ are distinct.

Therefore, $\mathcal{S}^2$ and $\mathcal R^2$ are the unique hyperplanes in $\mathcal{F}_{\mathcal{S}}$ and $\mathcal{F}_{\mathcal{R}}$, respectively, that are not a part of  triple intersection.  Since $\phi$ preserves the triple intersections, we conclude that $\phi(\mathcal{S}^2) = \mathcal{R}^2$.
\end{proof}

\begin{lemma}\label{no monomorphism}
Let $\phi \colon H^*(\mathbb{S}^m \times \mathbb{F}G(\mu); \mathbb{Q}) \to H^*(\mathbb{S}^m \times \mathbb{F}G(\nu); \mathbb{Q})$ be a graded $\mathbb{Q}$-algebra homomorphism, where $ \mathbb{F}G(\nu),\mathbb{F}G(\mu)$ have the same dimension, $\nu$ is not a permutation of $\mu$, and $(|\nu|, |\mu|) \neq (2,2)$.  Then $\phi$ is not a monomorphism.
\end{lemma}

\begin{proof}
First we prove for $\mathbb F=\mathbb C.$

Let $\nu = (n_1, \ldots, n_\ell)$ and $\mu = (m_1, \ldots, m_\wp)$. We follow the notations and observations from Notation~\ref{notation of hyperplanes}.
Note that $\dim \mathcal{R}^2 = \ell - 1$ and $\dim \mathcal{S}^2 = \wp - 1$. The linear map $\phi \colon \widetilde{\mathcal{S}^2} \to \widetilde{\mathcal{R}^2}$ cannot be a monomorphism if $\wp > \ell$. Thus, a monomorphism $\phi$ can only exist if $\wp \le \ell$, which we assume for the rest of the proof. 

We proceed by analyzing two cases: (i) $m \neq 2$, and (ii) $m = 2$.

\medskip
\noindent\textbf{Case (i): $m \neq 2$.} 
This case is analogous to the study of the nonexistence of monomorphisms between rational cohomology algebras of same-dimensional complex partial flag manifolds (see \cite{Man26}). For completeness, we include a self-contained argument.

By Corollary~\ref{kahler class descripition}, the subsets of elements in $\mathcal{S}^2$ and $\mathcal{R}^2$ of heights $<D$ are
\[
B := \bigcup_{1 \le p < q \le \wp} \mathcal{S}^2_{p,q} \quad\text{and}\quad A := \bigcup_{1 \le p < q \le \ell} \mathcal{R}^2_{p,q}\quad \text{respectively}.
\]
Thus all the elements of $\mathcal R^2\setminus A$ and $\mathcal S^2\setminus B$ have height $D.$
Because any graded algebra monomorphism preserves heights, assuming $\phi$ is a monomorphism forces $\phi(\mathcal{S}^2 \setminus B) \subseteq \mathcal{R}^2 \setminus A$. Applying Lemma~\ref{T(e_i)} and Lemma~\ref{extension}, it follows that $\wp = \ell$, and there exist a permutation $\sigma$ of $\{1, 2, \ldots, \wp\}$ and a nonzero scalar $\lambda \in \mathbb{Q} \setminus \{0\}$ such that
\begin{equation}\label{phi permutes}
\phi(y_i) = \lambda x_{\sigma(i)} \quad \text{for all } i = 1, 2, \ldots, \wp.
\end{equation}
By Theorem~\ref{heights of deg 2 elements}, elements in $\mathcal{S}^2_{i,j} \setminus \bigcup_{(p,q) \neq (i,j)} \mathcal{S}^2_{p,q}$ have height $D- m_i m_j$, while those in $\mathcal{R}^2_{i,j} \setminus \bigcup_{(p,q) \neq (i,j)} \mathcal{R}^2_{p,q}$ have height $D - n_i n_j$. Since $\phi$ maps $\mathcal{S}^2_{i,j}$ onto $\mathcal{R}^2_{\sigma(i),\sigma(j)}$ and preserves heights, \eqref{phi permutes} yields $D - m_i m_j = D - n_{\sigma(i)} n_{\sigma(j)}$, implying
\[
m_i m_j = n_{\sigma(i)} n_{\sigma(j)} \quad \text{for all } 1 \le i < j \le \wp.
\]

To show that $\mu$ and $\nu$ agree up to permutation, select three distinct indices $i, j, k \in \{1, 2, \ldots, \wp\}$ (which exist since $\wp \ge 3$). The system of pairwise relations gives
\[
m_i m_j = n_{\sigma(i)} n_{\sigma(j)}, \quad m_j m_k = n_{\sigma(j)} n_{\sigma(k)}, \quad \text{and} \quad m_i m_k = n_{\sigma(i)} n_{\sigma(k)}.
\]
Multiplying the first and third equations yields $m_i^2 (m_j m_k) = n_{\sigma(i)}^2 (n_{\sigma(j)} n_{\sigma(k)})$. Substituting the second relation into this product yields
\[
m_i^2 (n_{\sigma(j)} n_{\sigma(k)}) = n_{\sigma(i)}^2 (n_{\sigma(j)} n_{\sigma(k)}).
\]
This gives gives $m_i^2 = n_{\sigma(i)}^2$, which forces $m_i = n_{\sigma(i)}$. As $i$ was arbitrary, $m_i = n_{\sigma(i)}$ holds for every $1 \le i \le \wp$, contradicting the assumption that  that $\nu$ is not a permutation of $\mu$. Thus, $\phi$ cannot be a monomorphism.

\medskip
\noindent\textbf{Case (ii): $m = 2$.} 
By Lemma~\ref{heights of deg 2 elements} and the observation \eqref{heights for deg two for products}, the sets of cohomology classes not attaining maximal height $D+1$ in $\widetilde{\mathcal{S}^2}$ and $\widetilde{\mathcal{R}^2}$ are given respectively by
\[
\widetilde{B} \cup \mathcal{S}^2 \quad \text{and}\quad \widetilde{A} \cup \mathcal{R}^2, \quad \text{where } \widetilde{B} := \bigcup_{1 \le p < q \le \wp} \widetilde{\mathcal{S}^2_{p,q}} \;\text{ and }\; \widetilde{A} := \bigcup_{1 \le p < q \le \ell} \widetilde{\mathcal{R}^2_{p,q}}.
\]
If $\phi$ is a monomorphism, it preserves the set of cohomology classes of maximal heights:
\[
\phi\big(\widetilde{\mathcal{S}^2} \setminus (\widetilde{B} \cup \mathcal{S}^2)\big) \subseteq \widetilde{\mathcal{R}^2} \setminus (\widetilde{A} \cup \mathcal{R}^2).
\]
Taking complements gives the inclusion
\begin{equation}\label{hyperplane containment}
\phi(\widetilde{B} \cup \mathcal{S}^2) \supseteq \widetilde{A} \cup \mathcal{R}^2.
\end{equation}

Note that $\mathcal{S}^2$ and $\mathcal{R}^2$ are hyperplanes in $\widetilde{\mathcal{S}^2}$ and $\widetilde{\mathcal{R}^2}$, respectively. Since $\phi$ is a monomorphism, the image $\phi(\widetilde{B} \cup \mathcal{S}^2)$ is a union of hyperplanes in $\phi(\widetilde{\mathcal{S}^2})$. Comparing dimensions of the hyperplanes in $\widetilde{\mathcal S^2}$ and $\widetilde{\mathcal R^2}$, the inclusion \eqref{hyperplane containment} forces $\wp \ge \ell$, which, combined with $\wp \le \ell$, implies $\wp = \ell$. Since $(|\nu|, |\mu|)\neq (2,2)$, we have $\wp=\ell\ge 3.$

Consequently, both unions $\widetilde{B} \cup \mathcal{S}^2$ and $\widetilde{A} \cup \mathcal{R}^2$ consist of the same number of hyperplanes, namely $\binom{\wp}{2} + 1$. It follows that $\phi(\widetilde{B}\cup \mathcal S^2)=\widetilde{A}\cup \mathcal R^2.$ Thus, the linear monomorphism $\phi$ induces a bijection between these sets of hyperplanes:
\begin{equation}\label{bijection}
\{\widetilde{\mathcal{S}_{p,q}^2} : 1 \le p < q \le \wp\} \cup \{\mathcal{S}^2\} \; \longleftrightarrow \; \{\widetilde{\mathcal{R}_{p,q}^2} : 1 \le p < q \le \wp\} \cup \{\mathcal{R}^2\}.
\end{equation}
Applying the Lemma~\ref{tripple intersection}, we can conclude that $\phi(\mathcal S^2)=\mathcal R^2.$
The situation then reduces to Case (i). Following the same reasoning, it follows that $\phi$ cannot be a monomorphism.

The proof for $\mathbb F=\mathbb H$ is similar, with the proof divided into two cases according as $m\neq 4$ or $m=4$. Thus we omit the details.
\end{proof}

\begin{theorem}\label{deg zero, one not grassmannian}
    Let $\nu$ and $\mu$ be two finite sequences of natural numbers such that $\nu$ is not a permutation $\mu$, one of them has length at least $3$, and $\dim \mathbb CG(\nu)=\dim \mathbb CG(\mu).$ Then any continuous maps $f: \mathbb S^m\times \mathbb FG(\nu)\longrightarrow  \mathbb S^m\times \mathbb FG(\mu)$ has degree zero. Consequently, when $P(m,\nu)$ and $P(m,\mu)$ are orientable, for any continuous map $g:P(m,\nu)\longrightarrow P(m,\mu)$, the degree $\deg (g)$ is zero.
\end{theorem}
\begin{proof}
One completes the proof similarly as the proof of Proposition \ref{degree zero, m r are distinct}, using Remark~\ref{covering lift}, Lemma~\ref{deg zero iff not injective}, and Lemma~\ref{no monomorphism}.
\end{proof}

\subsection{Degrees of maps from $P(m,\mu)$ to $P(m, \nu)$, when $(|\nu|, |\mu|)= (2,2)$}
Now we proceed to study the existence of nonzero degree maps between generalized Dold manifolds fibred by distinct complex Grassmannians over the same real projective space.  In this context, the following results from the literature will be useful.

\begin{theorem}[\cite{SS09}, Theorem 1.2]\label{sarkar-sankaran degree}
Let $\mathbb F \in \{\mathbb C,\mathbb H\}$ and let $f \colon \mathbb FG(n_1,n_2) \to \mathbb FG(m_1,m_2)$ be a continuous map between two Grassmannians. Assume that 
\begin{enumerate}
    \item $n_1n_2=m_1m_2=$ the $\mathbb F$-dimension of the Grassmannians,
    \item $2 \le \min\{m_1,m_2\} < \min\{n_1,n_2\}$ and
    \item $(n_1^2-1)(n_2^2-1)(m_1^2-1)(m_2^2-1)$ is not a perfect square.
\end{enumerate}
Then the Brouwer degree of $f$ is zero.
\end{theorem}

\begin{theorem}[\cite{GH81}, Theorem 1.2]\label{glover-homer realization}
Let $X$ and $Y$ be  formal, nilpotent, finite CW complexes. For any graded algebra homomorphism $\phi \colon H^*(X;\mathbb Q) \to H^*(Y;\mathbb Q)$, there exists an automorphism $f^*$ of $H^*(Y;\mathbb Q)$ induced by a continuous self-map $f \colon Y \to Y$ such that $f^* \circ \phi = g^*$ for some continuous map $g \colon X \to Y$.
\end{theorem}

The following lemma will be useful.

\begin{lemma}\label{no mono for Grassmannians}
Let $\mathbb F \in \{\mathbb C, \mathbb H\}$ and let $\phi \colon H^*(\mathbb FG(m_1,m_2);\mathbb Z) \to H^*(\mathbb FG(n_1,n_2);\mathbb Z)$ be a graded algebra homomorphism between the cohomology algebras of equal dimensional Grassmannians. Suppose that one of the following holds:
\begin{enumerate}
    \item $1 \le \min\{n_1,n_2\} < \min \{m_1,m_2\}$;
    \item $2 \le \min\{m_1,m_2\} < \min \{n_1, n_2\}$ and\\ $(n_1^2-1)(n_2^2-1)(m_1^2-1)(m_2^2-1)$ is not a perfect square.
\end{enumerate}
Then $\phi$ is not a monomorphism.
\end{lemma}

\begin{proof}
Without loss of generality assume $n_1\le n_2$ and $m_1\le m_2$; recall also that $n_1n_2=m_1m_2$, since $\mathbb FG(n_1,n_2)$ and $\mathbb FG(m_1,m_2)$ have the same dimension.

We only prove for $\mathbb F=\mathbb C$ and the case for $\mathbb F=\mathbb H$ follows from a similar argument.

\medskip
\noindent\textbf{Case (1).} Here $n_1<m_1$. Combined with $n_1n_2=m_1m_2$ and $m_1\le m_2$, this forces
\[
n_2=\frac{m_1m_2}{n_1}>\frac{m_1m_2}{m_1}=m_2,
\]
so that $1\le n_1<m_1\le m_2<n_2$. The argument below follows the same strategy as \cite[Theorem 4.2, Case $r=2$]{Man26}; we include the details for completeness.

Recall from \eqref{eq:grassmann-reduced} the integral cohomology presentations
\[
H^*(\mathbb CG(n_1,n_2);\mathbb Z)\cong \mathbb Z[a_1,\dots,a_{n_1}]/I_{n_1,n_2}, \quad
H^*(\mathbb CG(m_1,m_2);\mathbb Z)\cong \mathbb Z[b_1,\dots,b_{m_1}]/I_{m_1,m_2},
\]
where $a_i$ (resp. $b_i$) is the $i$-th Chern class of the canonical rank-$n_1$ (resp. rank-$m_1$) bundle, $\deg a_i=\deg b_i=2i$, and the ideal of relations $I_{n_1,n_2}$ (resp. $I_{m_1,m_2}$) is generated in degrees $\ge 2(n_2+1)$ (resp. $\ge 2(m_2+1)$).

Since $n_1+1\le n_2$, we have $2(n_1+1)<2(n_2+1)$. Thus the rank of $H^{2(n_1+1)}(\mathbb CG(n_1, n_2);\mathbb Z)$ equals the number of monomials $a_1^{e_1}\cdots a_{n_1}^{e_{n_1}}$ with $\sum_i i\,e_i=n_1+1$, i.e. the number of partitions of $n_1+1$ into at most $n_1$ parts. As $n_1+1<m_2+1$ as well, the same reasoning applies to $H^{2(n_1+1)}(\mathbb CG(m_1,m_2);\mathbb Z)$, whose rank equals the number of monomials in $b_1,\dots,b_{m_1}$ of total degree $n_1+1$. Because $m_1\ge n_1+1$, every monomial that only involves $a_1,\dots,a_{n_1}$ has an obvious counterpart in $b_1,\dots,b_{n_1}$, and in addition the single generator $b_{n_1+1}$ itself contributes a monomial of degree $2(n_1+1)$ that has no analogue in terms of $a_i$'s. Hence
\[
\operatorname{rank} H^{2(n_1+1)}(\mathbb CG(n_1,n_2);\mathbb Z) \;<\; \operatorname{rank} H^{2(n_1+1)}(\mathbb CG(m_1,m_2);\mathbb Z).
\]
This prevents $\phi$ from being a graded monomorphism.

\medskip
\noindent\textbf{Case (2).} Since the integral cohomology of a complex Grassmannian is torsion-free, it suffices to show that there is no graded algebra monomorphism
\[
H^*(\mathbb CG(m_1, m_2);\mathbb Q)\longrightarrow H^*(\mathbb CG(n_1, n_2);\mathbb Q).
\]
Suppose, for contradiction, that such a monomorphism $\phi$ exists. By Theorem \ref{glover-homer realization}, there is an automorphism $f^*$ of $H^*(\mathbb CG(n_1,n_2);\mathbb Q)$, induced by a self-map $f$ of $\mathbb CG(n_1,n_2)$, such that $f^*\circ\phi$ is induced by a continuous map
\[
g:\mathbb CG(n_1,n_2)\to \mathbb CG(m_1,m_2).
\]
Since $f^*$ is an automorphism, $g^*=f^*\circ\phi$ is again a monomorphism of graded algebras.

Since any nonzero class in $H^2(\mathbb{C}G(n_1,n_2); \mathbb{Q})$ (resp., $H^2(\mathbb{C}G(m_1,m_2); \mathbb{Q})$) is a K\"ahler class, its top power yields a nonzero multiple of the fundamental cohomology class. As $g^*$ is injective, it maps a degree-$2$ class to a nonzero K\"ahler class, implying $g^*$ is nonzero on top cohomology. Consequently, the Brouwer degree $\deg(g) \neq 0$, which contradicts Theorem \ref{sarkar-sankaran degree}. Thus, no such algebra monomorphism $\phi$ exists.

This completes the proof.
\end{proof}

\begin{remark}\label{no mono in products}
    Continuing the discussion in the proof of Case (1) of Lemma~\ref{no mono for Grassmannians}, observe that for any $i < d(n_1+1)$, where $d = \dim_{\mathbb R}\mathbb F$,
    \begin{equation}\label{rank equility in low degree}
        \mathrm{rank}\, H^i(\mathbb FG(n_1, n_2);\mathbb Z)\; = \;\mathrm{rank}\, H^i(\mathbb FG(m_1, m_2);\mathbb Z).
    \end{equation}
     We have a strict inequality
    \begin{equation}\label{rank strict inequality}
        \operatorname{rank} H^{d(n_1+1)}(\mathbb FG(n_1,n_2);\mathbb Z) \;<\; \operatorname{rank} H^{d(n_1+1)}(\mathbb FG(m_1,m_2);\mathbb Z).
    \end{equation}
    By Poincar\'e duality, any discrepancy in the ranks of the cohomology groups can only occur in the range $[d(n_1+1), D - d(n_1+1)]$, where $D$ denotes the common real dimension of the Grassmannians $\mathbb FG(n_1, n_2)$ and $\mathbb FG(m_1, m_2)$. 
    Furthermore, this strict rank inequality carries over when taking product spaces with a sphere $\mathbb{S}^m$. By the Künneth formula, the rank of $j$-th cohomology of $X \times \mathbb{S}^m$ is given by the sum $(\mathrm{rank}\, H^j(X;\mathbb{Z}) + \mathrm{rank}\, H^{j-m}(X;\mathbb{Z}))$. Hence, by \eqref{rank equility in low degree} and \eqref{rank strict inequality}, we obtain 
    \[
    \operatorname{rank} H^{d(n_1+1)}\big(\mathbb S^m \times \mathbb FG(n_1,n_2);\mathbb Z\big) \;<\; \operatorname{rank} H^{d(n_1+1)}\big(\mathbb S^m \times \mathbb FG(m_1,m_2);\mathbb Z\big).
    \]
    Thus, it follows that there is no monomorphism $\phi$ from $H^*(\mathbb S^m \times \mathbb FG(m_1, m_2);\mathbb Z)$ to $H^*(\mathbb S^m \times \mathbb FG(n_1, n_2);\mathbb Z),$ if $\min{\{n_1,n_2\}}< \min \{m_1, m_2\}.$
    
\end{remark}

\begin{proposition}\label{deg zero, min ni is less}
Let $\nu=(n_1,n_2)$ and $\mu=(m_1,m_2)$ be tuples of natural numbers satisfying $n_1n_2=m_1m_2$ and let $\mathbb F\in\{\mathbb C,\mathbb H\}$. Suppose that $\min\{n_1,n_2\}<\min\{m_1,m_2\}.$
Then every continuous map $f:\mathbb S^m\times\mathbb FG(\nu)
\longrightarrow
\mathbb S^m\times\mathbb FG(\mu)$
has degree zero.  Consequently, when $P(m,\nu)$ and $P(m,\mu)$ are orientable, for any continuous map $g:P(m,\nu)\longrightarrow P(m,\mu)$, its degree $\deg(g)$ is zero.
\end{proposition}

\begin{proof}
The proof follows similarly as the proof of Proposition \ref{degree zero, m r are distinct}, using Remark \ref{covering lift} and Lemma \ref{no mono in products}.
\end{proof}

\section{Cohomological rigidity of the products $\mathbb S^m \times \mathbb CG(\nu)$}\label{sec cohomological rigidity}

In this section, we apply the observations of Section~\ref{nonexistance on nonzero deg} to prove the cohomological rigidity of products of the form $\mathbb S^m\times\mathbb CG(\nu)$.

Our study of the existence of nonzero-degree maps between products of a complex or quaternionic partial flag manifold and a sphere required us to rule out graded algebra monomorphisms between the corresponding rational cohomology algebras. Since a graded isomorphism induces monomorphisms in both directions, for any two distinct products $\mathbb S^m \times \mathbb FG(\nu)$ and $\mathbb S^r \times \mathbb FG(\mu)$ of the same dimension, if a monomorphism in one direction does not exist, then their rational cohomology algebras cannot be isomorphic. In view of this, as a consequence of Lemma~\ref{no mono m r distinct}, Lemma~\ref{no monomorphism}, and Remark~\ref{no mono in products}, we obtain the following cohomological rigidity result.

\begin{theorem}\label{thm cohomological rigidity}
    Let $\mathbb F\in \{\mathbb C, \mathbb H\}.$ Let there be a graded algebra isomorphism between the cohomology algebras of the products $H^*(\mathbb S^m \times \mathbb FG(\nu);\mathbb Q)$ and $H^*(\mathbb S^r\times \mathbb FG(\mu);\mathbb Q)$. Then $m=r$ and $\nu$ is a permutation of $\mu.$ Consequently, the products are homeomorphic.
\end{theorem}

The generalized Dold manifolds under consideration in the article have fundamental group isomorphic to $\mathbb Z_2$ or $\mathbb Z$ and are therefore not simply connected. Thus, an isomorphism $f^*$ of their rational cohomology algebras, induced from a continuous map $f$, does not allow us to invoke Whitehead's theorem directly to conclude that $f$ is a homotopy equivalence. But, under certain additional hypothesis, we can conclude stronger, that is, the GDMs are homeomorphic.
In this context, we prove the following proposition.
\begin{proposition}\label{cohom rigidity of GDM}
Let $\nu=(n_1, \ldots, n_\ell)$ and $\mu=(m_1, \ldots, m_\wp)$. Consider a continuous map
$f\colon P(m,\nu)\longrightarrow P(r,\mu)$ between two oriented, same dimensional GDMs, such that
\[
f^*\colon H^*(P(r,\mu);\mathbb Q)
\longrightarrow H^*(P(m,\nu);\mathbb Q)
\]
is an isomorphism and the following hold:
\begin{enumerate}
    \item  $(m,\mu)\neq (2, (1,1)),$
    \item  $\min \{n_i\}_i< \min \{m_i\}_i$, when $(\ell, \wp)=(2,2).$
\end{enumerate}
Then $m=r$ and $\nu$ is a permutation of
$\mu$. In particular, the GDMs $P(m,\nu)$ and $P(r,\mu)$ are homeomorphic.
\end{proposition}
\begin{proof}
    The cohomology isomorphism implies that $f^*$ is an isomorphism in top degree. Hence, by the definition of the Brouwer degree, $\deg(f)\neq 0.$ Applying Proposition~\ref{degree zero, m r are distinct} along with the assumption $(m,\mu)\neq (2, (1,1))$, implies that $m$ and $r$ cannot be distinct. Thus we have $m=r.$ 
    
    Now observe the two cases:  $(\ell, \wp)=(2,2)$ and $(\ell, \wp)\neq (2,2).$

    When $(\ell, \wp)=(2,2)$, using the assumption $\min \{n_i\}_i< \min \{m_i\}_i$, Proposition~\ref{deg zero, min ni is less} implies that $\deg(f)=0,$ which is a contradiction. Thus, this case  does not arise. 

    When $(\ell, \wp)\neq (2,2),$ Theorem~\ref{deg zero, one not grassmannian} implies that the only way $\deg(f)\neq 0$ is when $\nu$ is a permutation of $\mu.$

    Therefore, we have $m=r$ and $\nu=\mu$ up to a permutation. Consequently, $P(m,\nu)$ and $P(r, \mu)$ are homeomorphic. 
\end{proof}

\begin{remark}
Note that Proposition~\ref{cohom rigidity of GDM} covers all cases with $m\neq r$, except for one exceptional case $(m,\mu)=(2,(1,1))$. In the
case $m=r$, it covers all cases in which at least one of the fibers of the domain and codomain is not a Grassmannian, that is, $(\ell, \wp)\neq (2,2)$. Of course, all non-orientable GDMs are excluded in the proposition.
\end{remark}




\begin{thebibliography}{BHH83}

\bibitem[BHH83]{BHH83}
S.~A. Broughton, M.~Hoffman, and W.~Homer.
\newblock The height of two-dimensional cohomology classes of complex flag
  manifolds.
\newblock {\em Canad. Math. Bull.}, 26(4):498--502, 1983.

\bibitem[Bor53]{Bor53}
A.~Borel.
\newblock Sur la cohomologie des espaces fibr\'es principaux et des espaces
  homog\`enes de groupes de {L}ie compacts.
\newblock {\em Ann. of Math. (2)}, 57:115--207, 1953.

\bibitem[Bro11]{Bro11}
L.~E.~J. Brouwer.
\newblock \"uber {A}bbildung von {M}annigfaltigkeiten.
\newblock {\em Math. Ann.}, 71(1):97--115, 1911.

\bibitem[CMS10]{CMS10}
S.~Choi, M.~Masuda, and D.~Y. Suh.
\newblock Topological classification of generalized {B}ott towers.
\newblock {\em Trans. Amer. Math. Soc.}, 362(2):1097--1112, 2010.

\bibitem[Dol56]{Do56}
A.~Dold.
\newblock Erzeugende der {T}homschen {A}lgebra {\protect$\mathfrak{n}$}.
\newblock {\em Math. Z.}, 65:25--35, 1956.

\bibitem[Ful98]{Ful98}
W.~Fulton.
\newblock {\em Intersection theory}, volume~2 of {\em Ergebnisse der Mathematik
  und ihrer Grenzgebiete. 3. Folge}.
\newblock Springer-Verlag, Berlin, 2nd edition, 1998.

\bibitem[GH78]{GH78}
H.~Glover and B.~Homer.
\newblock Endomorphisms of the cohomology ring of finite {G}rassmann manifolds.
\newblock In {\em Geometric applications of homotopy theory ({P}roc. {C}onf.,
  {E}vanston, {I}ll., 1977), {I}}, volume 657 of {\em Lecture Notes in Math.},
  pages 170--193. Springer, Berlin-New York, 1978.

\bibitem[GH81]{GH81}
H.~H. Glover and W.~D. Homer.
\newblock Self-maps of flag manifolds.
\newblock {\em Trans. Amer. Math. Soc.}, 267(2):423--434, 1981.

\bibitem[Gro82]{Gro82}
M.~Gromov.
\newblock Volume and bounded cohomology.
\newblock {\em Inst. Hautes \'Etudes Sci. Publ. Math.}, (56):5--99, 1982.

\bibitem[Gro99]{Gro99}
M.~Gromov.
\newblock {\em Metric Structures for Riemannian and Non-Riemannian Spaces},
  volume 152 of {\em Progress in Mathematics}.
\newblock Birkh\"auser, Boston, 1999.

\bibitem[HH84]{HH84}
M.~Hoffman and W.~Homer.
\newblock On cohomology automorphisms of complex flag manifolds.
\newblock {\em Proc. Amer. Math. Soc.}, 91(4):643--648, 1984.

\bibitem[HK22]{HK22}
A.~Higashitani and K.~Kurimoto.
\newblock Cohomological rigidity for {F}ano {B}ott manifolds.
\newblock {\em Math. Z.}, 301(3):2369--2391, 2022.

\bibitem[Man24]{Man24}
M.~Mandal.
\newblock {\em Cohomology of generalized Dold manifolds}.
\newblock Ph.d. thesis, Homi Bhabha National Institute, The Institute of
  Mathematical Sciences, 2024.

\bibitem[Man26]{Man26}
M.~Mandal.
\newblock Degrees of maps and cohomological rigidity of partial flag manifolds,
  2026.
\newblock arXiv:2608.09183.

\bibitem[MS74]{MS74}
J.~W. Milnor and J.~D. Stasheff.
\newblock {\em Characteristic classes}, volume No. 76 of {\em Annals of
  Mathematics Studies}.
\newblock Princeton University Press, Princeton, NJ; University of Tokyo Press,
  Tokyo, 1974.

\bibitem[MS22]{MS22}
M.~Mandal and P.~Sankaran.
\newblock Cohomology of generalized {D}old spaces.
\newblock {\em Topology Appl.}, 310:Paper No. 108040, 16, 2022.

\bibitem[MS24]{MS24}
M.~Mandal and P.~Sankaran.
\newblock Cohomology and {K}-theory of generalized dold manifolds fibred by
  complex flag manifolds, 2024.
\newblock To appear in \emph{Osaka J. Math.}; arXiv:2407.03932.

\bibitem[MS26]{MS26}
M.~Mandal and D.~Setia.
\newblock Rigidity of cohomology automorphisms of homogeneous spaces and
  coincidence theory, 2026.
\newblock To appear in \emph{Proc. Roy. Soc. Edinburgh Sect. A};
  arXiv:2512.21305.

\bibitem[NS19]{NS19}
A.~Nath and P.~Sankaran.
\newblock On generalized {D}old manifolds.
\newblock {\em Osaka J. Math.}, 56(1):75--90, 2019.

\bibitem[NS21]{NS21}
A.~Nath and P.~Sankaran.
\newblock A note on the equivariant cobordism of generalized {D}old manifolds.
\newblock {\em Topology Appl.}, 293:Paper No. 107554, 8, 2021.

\bibitem[PS89]{PS89}
K.~H. Paranjape and V.~Srinivas.
\newblock Self-maps of homogeneous spaces.
\newblock {\em Invent. Math.}, 98(2):425--444, 1989.

\bibitem[PS16]{PS16}
J.~Popko and A.~Szczepa\'nski.
\newblock Cohomological rigidity of oriented {H}antzsche-{W}endt manifolds.
\newblock {\em Adv. Math.}, 302:1044--1068, 2016.

\bibitem[RS97]{RS97}
V.~Ramani and P.~Sankaran.
\newblock On degrees of maps between {G}rassmannians.
\newblock {\em Proc. Indian Acad. Sci. Math. Sci.}, 107(1):13--19, 1997.

\bibitem[Spa66]{Spa66}
E.~H. Spanier.
\newblock {\em Algebraic topology}.
\newblock McGraw-Hill Book Co., New York-Toronto-London, 1966.

\bibitem[SS09]{SS09}
P.~Sankaran and S.~Sarkar.
\newblock Degrees of maps between {G}rassmann manifolds.
\newblock {\em Osaka J. Math.}, 46(4):1143--1161, 2009.

\end{thebibliography}

\end{document}